\documentclass{amsart}

\usepackage{xcolor}
\definecolor{lgray}{RGB}{240,240,240}
\definecolor{webgreen}{rgb}{0,.5,0}
\definecolor{webbrown}{rgb}{.6,0,0}
\definecolor{RoyalBlue}{cmyk}{1, 0.50, 0, 0}
\usepackage[colorlinks=true, breaklinks=true, urlcolor=webbrown, linkcolor=RoyalBlue, citecolor=webgreen, backref=page]{hyperref}

\usepackage{epsfig, graphicx, subfigure}
\usepackage{verbatim, setspace}
\usepackage{amsmath, amssymb}

\usepackage{newtxtext,newtxmath}

\usepackage{mathabx}

\newcommand{\C}		{\mathbb{C}}

\newcommand{\supp}{\mathrm{supp}}

\newcommand{\im}{\mathrm{Im}}

\renewcommand{\det}{\mathrm{det}}
\newcommand{\ic}{\mathrm{i}}

\newcommand{\qasq}{\quad \text{as} \quad}
\newcommand{\qandq}{\quad \text{and} \quad}

\newcommand{\vb}[1]{{\left\vert\kern-0.25ex\left\vert\kern-0.25ex\left\vert #1 
    \right\vert\kern-0.25ex\right\vert\kern-0.25ex\right\vert}}

\newcommand{\rhy}   {\textnormal{RHP}-${\boldsymbol Y}$}
\newcommand{\rhx}   {\textnormal{\(\bar\partial\)RHP}-${\boldsymbol X}$}
\newcommand{\rhn}   {\textnormal{RHP}-${\boldsymbol N}$}
\newcommand{\rha}   {\textnormal{RHP}-${\boldsymbol A}$}

\newcommand{\pbd}   {\textnormal{\(\bar\partial\)P}-${\boldsymbol D}$}

\usepackage[most]{tcolorbox}

\tcbsetforeverylayer{autoparskip}

\usepackage{varwidth}
\tcbuselibrary{skins}
\tcbuselibrary{theorems}

\newenvironment{myitemize}{\begin{itemize}}{\end{itemize}}
\tcolorboxenvironment{myitemize}{blanker, breakable, before skip=6pt, after skip=6pt, borderline west={1mm}{0pt}{RoyalBlue}}

\newtcolorbox{mathbox}[1][]{reset, enhanced, tcbox raise base, colback=black!5, colframe=webbrown, boxrule=0.2mm,after=,#1}

\newtcbtheorem[]{theorem}{Theorem}{enhanced,
before skip=2mm,after skip=2mm, boxrule=0.2mm, colback=black!5,colframe=webbrown,  colbacktitle=webbrown, breakable,  attach boxed title to top left={xshift=1cm,yshift*=1mm-\tcboxedtitleheight}, varwidth boxed title*=-3cm,
boxed title style={frame code={
            \path[fill=tcbcolback!30!black]
              ([yshift=-1mm,xshift=-1mm]frame.north west)
                arc[start angle=0,end angle=180,radius=1mm]
              ([yshift=-1mm,xshift=1mm]frame.north east)
                arc[start angle=180,end angle=0,radius=1mm];
            \path[left color=tcbcolback!60!black,right color=tcbcolback!60!black,
              middle color=tcbcolback!80!black]
              ([xshift=-2mm]frame.north west) -- ([xshift=2mm]frame.north east)
              [rounded corners=1mm]-- ([xshift=1mm,yshift=-1mm]frame.north east)
              -- (frame.south east) -- (frame.south west)
              -- ([xshift=-1mm,yshift=-1mm]frame.north west)
              [sharp corners]-- cycle;
            },interior engine=empty,
          },
          fonttitle=\bfseries,
          title={#2},#1}{thm}

\newtcbtheorem[]{proposition}{Proposition}{enhanced, before skip=2mm,after skip=2mm, boxrule=0.2mm, colback=black!5,colframe=webbrown, colbacktitle=webbrown, breakable, attach boxed title to top left={xshift=1cm,yshift*=1mm-\tcboxedtitleheight}, varwidth boxed title*=-3cm,
boxed title style={frame code={
            \path[fill=tcbcolback!30!black]
              ([yshift=-1mm,xshift=-1mm]frame.north west)
                arc[start angle=0,end angle=180,radius=1mm]
              ([yshift=-1mm,xshift=1mm]frame.north east)
                arc[start angle=180,end angle=0,radius=1mm];
            \path[left color=tcbcolback!60!black,right color=tcbcolback!60!black,
              middle color=tcbcolback!80!black]
              ([xshift=-2mm]frame.north west) -- ([xshift=2mm]frame.north east)
              [rounded corners=1mm]-- ([xshift=1mm,yshift=-1mm]frame.north east)
              -- (frame.south east) -- (frame.south west)
              -- ([xshift=-1mm,yshift=-1mm]frame.north west)
              [sharp corners]-- cycle;
            },interior engine=empty,
          },
          fonttitle=\bfseries,
          title={#2},#1}{prop}

\newtcbtheorem[]{lemma}{Lemma}{enhanced, before skip=2mm,after skip=2mm, boxrule=0.2mm, colback=black!5,colframe=webbrown, colbacktitle=webbrown, breakable, attach boxed title to top left={xshift=1cm,yshift*=1mm-\tcboxedtitleheight}, varwidth boxed title*=-3cm,
boxed title style={frame code={
            \path[fill=tcbcolback!30!black]
              ([yshift=-1mm,xshift=-1mm]frame.north west)
                arc[start angle=0,end angle=180,radius=1mm]
              ([yshift=-1mm,xshift=1mm]frame.north east)
                arc[start angle=180,end angle=0,radius=1mm];
            \path[left color=tcbcolback!60!black,right color=tcbcolback!60!black,
              middle color=tcbcolback!80!black]
              ([xshift=-2mm]frame.north west) -- ([xshift=2mm]frame.north east)
              [rounded corners=1mm]-- ([xshift=1mm,yshift=-1mm]frame.north east)
              -- (frame.south east) -- (frame.south west)
              -- ([xshift=-1mm,yshift=-1mm]frame.north west)
              [sharp corners]-- cycle;
            },interior engine=empty,
          },
          fonttitle=\bfseries,
          title={#2},#1}{lem}

\begin{document}

\title{On Smooth Perturbations of Chebysh\"ev Polynomials and \( \bar\partial \)-Riemann-Hilbert Method}

\author{Maxim L. Yattselev}

\address{Department of Mathematical Sciences, Indiana University-Purdue University Indianapolis, 402~North Blackford Street, Indianapolis, IN 46202}

\address{Keldysh Institute of Applied Mathematics, Russian Academy of Science, Miusskaya Pl. 4, Moscow, 125047 Russian Federation}

\email{\href{mailto:maxyatts@iupui.edu}{maxyatts@iupui.edu}}

\thanks{The research was supported in part by a grant from the Simons Foundation, CGM-706591.}



\begin{abstract}
\( \bar\partial \)-extension of the matrix Riemann-Hilbert method is used to study asymptotics of the polynomials \( P_n(z) \) satisfying orthogonality relations
\[
\int_{-1}^1 x^lP_n(x)\frac{\rho(x)dx}{\sqrt{1-x^2}}=0, \quad l\in\{0,\ldots,n-1\},
\]
where \( \rho(x) \) is a positive \( m \) times continuously differentiable function on \( [-1,1] \), \( m\geq3 \).
\end{abstract}

\maketitle

\section{Main Results}

In this note we are interested in the asymptotic behavior of monic polynomials \( P_{n,i}(x) \),  \( \deg (P_{n,i})=n \), dependent on a parameter \( i\in\{1,2,3,4\} \), satisfying orthogonality relations
\begin{mathbox}[top=-1mm,bottom=1mm]
\begin{equation}
\label{ortho}
\int_{-1}^1 x^lP_{n,i}(x)\frac{\rho(x)|v_i(x)|dx}{\sqrt{1-x^2}} =0, \quad l\in\{0,\ldots,n-1\},
\end{equation}
\end{mathbox}
where \( \rho(x) \) is a positive and smooth function on \( [-1,1] \) and
\[
v_1(z) \equiv1, \quad v_2(z) = z^2-1, \quad v_3(z) = z+1, \qandq v_4(z) = z-1.
\]
That is, \(  P_{n,i}(z) \) are smooth perturbations of the Chebysh\"ev polynomials of the \( i \)-th kind. Besides polynomials themselves, we are also interested in the asymptotic behavior of their recurrence coefficients. That is, numbers \( a_{n,i}\in[0,\infty) \) and \( b_{n,i}\in(-\infty,\infty) \) such that
\begin{mathbox}[top=0mm,bottom=1mm]
\[
xP_{n,i}(x) = P_{n+1,i}(x) + b_{n,i} P_{n,i}(x) + a_{n,i}^2 P_{n-1,i}(x).
\]
\end{mathbox}

To describe the results, let \( w(z) := \sqrt{z^2-1} \) be the branch analytic in \( \C\setminus[-1,1] \) such that \( w(z)/z \to 1 \) as \( z\to\infty \).  The Szeg\H{o} function of the weight \( \rho(x) \) is defined by
\begin{mathbox}[top=-1mm,bottom=1mm]
\begin{equation}
\label{S}
S(z) := \exp\left\{\frac{w(z)}{2\pi\ic}\int_{-1}^1\frac{\log\rho(x)}{z-x}\frac{dx}{w_+(x)}\right\}, \quad z\in\overline\C\setminus[-1,1],
\end{equation}
\end{mathbox}
which is an analytic and non-vanishing function in the domain of its definition satisfying
\begin{mathbox}[top=-1mm,bottom=2mm]
\begin{equation}
\label{Sj}
S_+(x)S_-(x) = \rho^{-1}(x), \quad x\in[-1,1].
\end{equation}
\end{mathbox}
Since \( \rho(x) \) is positive, it holds that \( S_+(x) = \overline{S_-(x)} \) for \( x\in [-1,1] \), and, utilizing the full power of Plemelj-Sokhotski formulae, \eqref{Sj} can be strengthen to
\begin{mathbox}[top=-1mm,bottom=1mm]
\begin{equation}
\label{theta}
\sqrt{\rho(x)}S_\pm(x) = e^{\pm\ic\theta(x)}, \quad \theta(x) := \frac{\sqrt{1-x^2}}{2\pi} \fint_{-1}^1 \frac{\log\rho(t)}{t-x} \frac{dt}{\sqrt{1-t^2}},
\end{equation}
\end{mathbox}
where \( \fint \) is the integral in the sense of the principal value. Further, let
\begin{mathbox}[top=-2mm,bottom=1mm]
\begin{equation}
\label{phi}
\varphi(z) := z + w(z)
\end{equation}
\end{mathbox}
be the conformal map of \( \overline\C\setminus[-1,1] \) onto \( \C\setminus\{z:|z|\geq1\} \) such that \( \varphi(z)/z \to 2 \) as \( z\to\infty \). One can readily verify that
\begin{mathbox}[top=-1mm,bottom=1mm]
\begin{equation}
\label{phij}
\varphi_\pm(x) = x\pm\ic \sqrt{1-x^2} = e^{\pm\ic\arccos(x)}, \quad x\in[-1,1].
\end{equation}
\end{mathbox}
 Finally, we explicitly define the Szeg\H{o} functions of the weights \( |v_i(x)| \). Namely, set
\begin{mathbox}[top=-1mm,bottom=1mm]
\begin{equation}
\label{Si}
\begin{cases}
S_1(z) :\equiv 1, & S_3(z) := \big(\varphi(z)/(z+1)\big)^{1/2}, \medskip \\
S_2(z) := \varphi(z)/w(z), & S_4(z) := \big(\varphi(z)/(z-1)\big)^{1/2},
\end{cases}
\end{equation}
\end{mathbox}
where the square roots are principal and one needs to notice that the images of \( \overline\C\setminus[-1,1] \) under \( (z+1)/\varphi(z) \) and \( (z-1)/\varphi(z) \) are domains symmetric with respect to conjugation whose intersections with the real line are equal to \( (0,2) \) (so the square roots are indeed well defined). These functions satisfy
\begin{mathbox}[top=-1mm,bottom=2mm]
\begin{equation}
\label{Sij}
S_{i+}(x)S_{i-}(x) = |S_{i\pm}(x)|^2=1/|v_i(x)|, \quad x\in(-1,1).
\end{equation}
\end{mathbox}
Observe also that \( S_1(\infty) =1 \), \( S_2(\infty) = 2 \), and \( S_3(\infty) = S_4(\infty) =\sqrt2 \). Moreover, one can readily deduce from \eqref{phij} and \eqref{Sij} that
\begin{mathbox}[top=-1mm,bottom=1mm]
\begin{equation}
\label{thetai}
S_{i\pm}(x) = \frac{e^{\pm\ic\theta_i(x)}}{\sqrt{|v_i(x)|}}, \quad 
\begin{cases}
\theta_1(x) :\equiv0, & \theta_2(x) := \arccos(x)-\frac\pi2, \smallskip \\
\theta_3(x) := \frac12\arccos(x), & \theta_4(x) :=\frac12\arccos(x)-\frac\pi2.
\end{cases}
\end{equation}
\end{mathbox}

Recall that the modulus of continuity of a continuous function \( f(x) \) on \( [-1,1] \) is given by
\[
\omega(f;h) := \max_{|x-y|\leq h,~x,y\in[-1,1]} |f(x)-f(y)|.
\]

\begin{theorem}[]{}{1}
Assume that \( \rho(x) \) is a strictly positive \( m \) times continuously differentiable function on \( [-1,1] \) for some \( m \geq 3 \). Set
\[
\varepsilon_n:= \frac{\log n}{n^m}\omega\left((1/\rho)^{(m)};1/n\right).
\] 
Then it holds for any \( i\in\{1,2,3,4 \} \) that
\[
P_{n,i}(z) = \left( 1 + O(\varepsilon_n) \right) \frac{(S_iS)(z)}{(S_iS)(\infty)} \left(\frac{\varphi(z)}2\right)^n
\]
uniformly on closed subsets of \( \overline\C\setminus[-1,1] \) and
\[
P_{n,i}(x) = \frac{\cos\big(n\arccos(x) + \theta(x) + \theta_i(x) \big) + O(\varepsilon_n) }{2^{n-1}(S_iS)(\infty)\sqrt{\rho(x)|v_i(x)|}}
\]
uniformly on \( [-1,1] \). Moreover, it also holds for any \( i\in\{1,2,3,4\} \) that
\[
a_{n,i} = 1/2 + O(\varepsilon_n) \qandq b_{n,i} = O(\varepsilon_n).
\]
\end{theorem}

The above results are not entirely new. It is well known \cite[Theorem~11.5]{Szego} that perturbed first and second kind Chebysh\"ev polynomials can be expressed via orthogonal polynomials on the unit circle with respect to the weight \( \rho(\frac12(\tau+1/\tau)) \). Then using \cite[Corollary~5.2.3]{Simon}, that in itself is an extension of ideas from \cite{Bax61}, and Geronimus relations, see \cite[Theorem~13.1.7]{Simon}, one can show that
\[
\sum (n+1)^\gamma \big( |a_{n,1}-1/2| + |b_{n,1}| \big) < \infty
\]
for any \( \gamma\in(0,m-1) \) and \( m\geq 2 \), which is consistent with Theorem~\ref{thm:1}. What is novel in this note is the method of proof. While the Baxter-Simon argument relies on the machinery of Banach algebras, we follow the approach of Fokas, Its, and Kitaev \cite{FIK91,FIK92} connecting orthogonal polynomials to matrix Riemann-Hilbert problems and then utilizing the non-linear steepest descent method of Deift and Zhou \cite{DZ93}. The main advantages of this approach are the ability to get full asymptotic expansions for analytic weights of orthogonality \cite{DKMLVZ99b,KMcLVAV04} and its indifference to positivity of such weights \cite{Ap02,BerMo09,ApYa15}. However, here we deal with non-analytic densities by elaborating on the idea of extensions with controlled \( \bar\partial \)-derivative introduced by Miller and McLaughlin \cite{McLM08} and adapted to the setting of Jacobi-type polynomials by Baratchart and the author \cite{BYa10}.

\section{Weight Extension}

Given \( r>1 \), let \( E_r :=\{z:|\varphi(z)|<r\} \). The boundary \( \partial E_r \)  is an ellipse with foci \( \pm1 \).

\begin{proposition}[]{}{1}
Let \( \rho(x) \) and \( \varepsilon_n \) be as in Theorem~\ref{thm:1}. For each \( r>1 \) and \( n>2m \) there exists a continuous function \( \ell_{n,r}(z) = l_n(z) + L_{n,r}(z) \), \( z\in\C \), such that
\[
\ell_{n,r}(x) = l_n(x), \quad x\in[-1,1],
\] 
where \( l_n(z) \) is a polynomial of degree at most \( n \) satisfying
\[
\supp_{x\in[-1,1]}|l_n(x)| \leq C_\rho^\prime
\]
for some constant \( C_\rho^\prime \) independent of \( n \), while \( L_{n,r}(z) \) and \( \bar\partial L_{n,r}(z) \) are continuous functions in \( \C \) supported by \( \overline E_r \) (in particular, \( L_{n,r}(z)=0 \) for \( z\notin E_r\)) and
\[
\frac{|\bar\partial L_{n,r}(z)|}{\sqrt{|1-z^2|}} \leq C_\rho^{\prime\prime}\frac{n \varepsilon_n}{\log n}, \quad z\in \overline E_r,
\]
for some constant \( C_\rho^{\prime\prime} \) independent of \( n \) and \( r \), where \( \bar\partial := (\partial_x+\ic\partial_y)/2 \), \( z=x+\ic y \).
\end{proposition}
\begin{proof}
It follows from \cite[Theorem~9]{Kil00} that for each \( n >2m \) there exists a polynomial \( l_n(z) \) of degree at most \( n \) such that
\[
\left|\left(\rho^{-1}(x)\right)^{(k)} - l_n^{(k)}(x) \right| \leq C_{m,k} (1-x^2)^{\frac{m-k}2} n^{k-m} E_{n-m}\left(\left(\rho^{-1}\right)^{(m)}\right)
\]
for all \( x\in[-1,1] \) and each \( k\in\{0,\ldots,m\} \), where \( C_{m,k} \) is a constant that depends only \( m \) and \( k \) and \( E_j(f) \) is the error of best uniform approximation on the interval \( [-1,1] \) of a continuous function \( f(x) \) by algebraic polynomials of degree at most \( j \). Furthermore, it was shown by Timan, see \cite[Equation~(3)]{Kil00}, that
\begin{eqnarray*}
E_{n-m}(f) &\leq& C_1\omega\left(f;\frac{\sqrt{1-x^2}}{n-m} + \frac1{(n-m)^2}\right) \leq C_1\omega\left(f;\frac2{n-m}\right) \\
&\leq& C_1\omega\left(f;\frac4n\right) \leq 4C_1\omega\left(f;\frac1n\right)
\end{eqnarray*}
for some absolute constant \( C_1 \), where we used that \( n>2m \) and \( \omega(f;2h) \leq 2\omega(f;h)\) (in what follows, we understand that all constants \( C_j \) might depend on \( \rho(x) \), but are independent of \( n \)). Set
\[
\lambda_n(x) := \frac{\rho^{-1}(x)-l_n(x)}{\sqrt{1-x^2}}, \quad x\in[-1,1].
\]
It then holds that \( \lambda_n(x) \) is a continuous function on \( [-1,1] \) that satisfies \( \|\lambda_n\|\leq C_3 \varepsilon_n/\log n \), where \( \|\cdot\| \) is the uniform norm on \( [-1,1] \). Since \( m\geq3\), it also holds that
\[
\lambda_n^\prime(x) = \frac{\left(\rho^{-1}(x)\right)^\prime - l_n^\prime(x)}{\sqrt{1-x^2}} + x\frac{\rho^{-1}(x)-l_n(x)}{\sqrt{(1-x^2)^3}}
\]
is a continuous function on \( [-1,1] \) that satisfies \( \|\lambda_n^\prime\|\leq C_4n\varepsilon_n/\log n  \) (this is exactly the place where condition \( m\geq3 \) is used). Extend \( \lambda_n(x) \) by zero to the whole real line. As the numerator of \( \lambda_n(x) \) together with its first and second derivatives vanishes at \( \pm1 \), \( \lambda_n^\prime(x) \) also extends continuously by zero to the whole real line. The following construction  is standard, see \cite[Proof of Theorem~3.67]{DemengelDemengel}. Define
\[
\Lambda_n(z) := \frac1{|y|}\int_0^{|y|} \lambda_n(x+t)dt, \quad z=x+\ic y,
\]
which, due to continuity of \( \lambda_n(x) \), is a continuous function in \( \C \) satisfying \( \Lambda_n(x)=\lambda_n(x) \) on the real line and \( |\Lambda_n(z)|\leq \|\lambda_n\| \) in the complex plane. Similarly,
\[
\big|\partial_x\Lambda_n(z)\big| = \left|\frac1{|y|}\int_0^{|y|} \lambda_n^\prime(x+t)dt\right| \leq \|\lambda_n^\prime\|
\]
and the function \( \partial_x\Lambda_n(z) \), which is given by the integral within the absolute value in the above equation, is also continuous in \( \C \). Furthermore, we have that
\begin{eqnarray*}
\big|\partial_y\Lambda_n(z)\big| &=& \left| \frac1{y^2}\int_0^{|y|} \big( \lambda_n(x+t) - \lambda_n(x+|y|) \big) dt \right| \\
&\leq& \|\lambda_n^\prime\|\int_0^{|y|} \frac{|y|-t}{y^2}dt =  \frac{\|\lambda_n^\prime\|}2
\end{eqnarray*}
and is also a continuous function in \( \C \). Altogether, since \( \bar\partial = (\partial_x+\ic\partial_y)/2 \), it holds that \( \bar\partial \Lambda_n(z) \)  is a continuous function in \( \C \) that satisfies \( |\bar\partial \Lambda_n(z)| \leq \|\lambda_n^\prime\| \) in the complex plane. Let \( \psi_r(z) \) be any real-valued continuous function with continuous partial derivatives that is equal to one on \( [-1,1] \) and is equal to zero in the complement of \( E_r \). Define 
\[
L_{n,r}(z) := \ic w(z)\Lambda_n(z)\psi_r(z)
\begin{cases} 
-1, & \im(z)\geq0, \\
1, & \im(z)<0.
\end{cases}
\]
Since \( w_\pm(x)=\pm\ic\sqrt{1-x^2} \) for \( x\in[-1,1] \) and \( \Lambda_n(x)=0 \) for \( x\not\in(-1,1) \), it holds that \( L_{n,r}(z) \) is a continuous function in \( \C \) that is supported by \( \overline E_r \) and is equal to \( \rho^{-1}(x) - l_n(x) \) for \( x\in[-1,1] \). Furthermore, since \( \bar\partial(\Lambda_n(z)\psi_n(z)) \) is continuous in \( \C \) and vanishes for \( z=x\not\in(-1,1) \) while \( w_+(x)=-w_-(x) \) for \( x\in(-1,1) \),  \( \bar\partial L_{n,r}(z) \) is also continuous in \( \C \).  Moreover, it holds that
\begin{eqnarray*}
|\bar\partial L_{n,r}(z)| &=& \sqrt{|1-z^2|}~ \big|\bar\partial (\Lambda_n(z)\psi_r(z))\big | \\ 
& \leq & C_5\sqrt{|1-z^2|}~\big( |\Lambda_n(z)| + |\bar\partial\Lambda_n(z)| \big) \\
& \leq & C_6\sqrt{|1-z^2|}~\frac{n\varepsilon_n}{\log n}, \quad z\in \overline E_r.
\end{eqnarray*}
Finally, observe that polynomials \( l_n(x) \) approximate \( \rho^{-1}(x) \) on \( [-1,1] \) and therefore have uniformly bounded above uniform norms. The claim of the proposition now follows by setting \( \ell_{n,r}(z) := l_n(z) + L_{n,r}(z)  \) for \( l_n(z) \) and \( L_{n,r}(z) \) as above. 
\end{proof}

\section{Proof of Theorem~\ref{thm:1}}

\subsection{Initial Riemann-Hilbert Problem}

Notice that the functions \( v_i(x) \) and \( |v_i(x)| \) are either equal to each other or differ by a sign when \( x\in[-1,1] \). So, we can equally use \( v_i(x) \) in \eqref{ortho} without changing the polynomials \( P_{n,i}(x) \). 

Denote by \( R_{n,i}(z) \) the function of the second kind associated with \( P_{n,i}(z) \). That is,
\begin{equation}
\label{Rn}
R_{n,i}(z) := \frac1{2\pi\ic}\int_{-1}^1\frac{P_{n,i}(x)}{x-z}\frac{\rho(x)v_i(x)dx}{w_+(x)},
\end{equation} 
which is a holomorphic function in \( \overline\C\setminus[-1,1] \). It follows from Plemelj-Sokhotski formulae, \cite[Chapter I.4.2]{Gakhov}, that
\[
R_{n,i+}(x) - R_{n,i-}(x) = P_{n,i}(x)\frac{\rho(x)v_i(x)}{w_+(x)}, \quad x\in(-1,1),
\]
and, see \cite[Chapter I.8.4]{Gakhov}, that 
\[
R_{n,i}(z) = O\left( |z-a|^{\alpha_{a,i}} \right) \qasq \C\setminus[-1,1]\ni z\to a\in\{-1,1\},
\]
where \( \alpha_{a,i}=0 \) if \( v_i(a)=0 \) and \( \alpha_{a,i}=-1/2 \) otherwise. Moreover, we get from \eqref{ortho} that 
\[
R_{n,i}(z) = \frac1{m_{n,i}z^n} + O\left(\frac1{z^{n+1}}\right) \qasq z\to\infty
\]
for some finite constant \( m_{n,i} \). Consider the following Riemann-Hilbert problem for \( 2\times2 \) matrix functions (\rhy):
\begin{myitemize}
\label{rhy}
\item[(a)] \( \boldsymbol Y(z) \) is analytic in \( \C\setminus[-1,1] \) and \( \displaystyle \lim_{z\to\infty} \boldsymbol Y(z)z^{-n\sigma_3} = \boldsymbol I \);
\item[(b)] \( \boldsymbol Y(z) \) has continuous traces on \( (-1,1) \) that satisfy
\[
\displaystyle \boldsymbol Y_+(x) = \boldsymbol Y_-(x) \left(\begin{matrix} 1 & \frac{\rho(x)v_i(x)}{w_+(x)} \smallskip \\ 0 & 1 \end{matrix}\right);
\]
\item[(c)] \( \boldsymbol Y(z) \) behaves like
\[
\boldsymbol Y(z) = O\left(\begin{matrix} 1 & |z-a|^{\alpha_{a,i}} \smallskip \\ 1 & |z-a|^{\alpha_{a,i}}\end{matrix}\right) \qasq \C\setminus[-1,1]\ni z\to a\in\{-1,1\}.
\]
\end{myitemize}

The following lemma is well known \cite[Theorem~2.4]{KMcLVAV04}.

\begin{lemma}[]{}{1}
\hyperref[rhy]{\rhy} is uniquely solvable by
\begin{equation}
\label{Y}
\boldsymbol Y(z) = \left(\begin{matrix}
P_{n,i}(z) & R_{n,i}(z) \\
m_{n-1,i}P_{n-1,i}(z) & m_{n-1,i}R_{n-1,i}(z)
\end{matrix}\right).
\end{equation}
\end{lemma}

\subsection{Opening of the Lens}

Fix \( 1<r<R \) and orient \( \partial E_R \) clockwise. Set
\begin{equation}
\label{X}
\boldsymbol X(z) :=
\begin{cases}
\boldsymbol Y(z) \left(\begin{matrix} 1 & 0 \smallskip \\ -\frac{w(z)\ell_{n,r}(z)}{v_i(z)} & 1 \end{matrix}\right), & \mbox{in} \quad E_R\setminus[-1,1], \medskip \\
\boldsymbol Y(z), & \mbox{in} \quad \C\setminus\overline E_R,
\end{cases}
\end{equation}
where \( \ell_{n,r}(z) \) is the extension of \( \rho^{-1}(x) \) constructed in Proposition~\ref{prop:1}. Observe that
\[
\ell_{n,r}(s) = l_n(s), \;\; s\in \partial E_R, \qandq \bar\partial \ell_{n,r}(z) = \bar\partial L_{n,r}(z), \;\; z\in \overline E_r,
\]
 since \( L_{n,r}(z) \) is supported by \( \overline E_r \) and \( l_n(z) \) is analytic (in fact, is a polynomial). It is trivial to verify that \( \boldsymbol X(z) \) solves the following \( \bar\partial \)-Riemann-Hilbert problem (\rhx):
\begin{myitemize}
\label{rhx}
\item[(a)] \( \boldsymbol X(z) \) is continuous in \( \C\setminus([-1,1]\cup\partial E_R) \) and \( \lim_{z\to\infty} \boldsymbol X(z)z^{-n\sigma_3} = \boldsymbol I \);
\item[(b)] \( \boldsymbol X (z)\) has continuous traces on \( (-1,1)\cup \partial E_R\) that satisfy
\[
\boldsymbol X_+(s) =\boldsymbol X_- (s)\left\{
\begin{array}{rl}
\displaystyle \left(\begin{matrix} 0 & \frac{\rho(s)v_i(s)}{w_+(s)} \smallskip \\ -\frac{w_+(s)}{\rho(s)v_i(s)} & 0 \end{matrix}\right) & \text{on} \quad s\in (-1,1),  \medskip \\
\displaystyle \left(\begin{matrix} 1 & 0 \smallskip \\ \frac{w(s)l_n(s)}{v_i(s)} & 1 \end{matrix}\right) & \text{on} \quad s\in \partial E_R;
\end{array}
\right.
\]
\item[(c)] \( \boldsymbol X(z) \) has the same behavior near \( \pm1 \) as \( \boldsymbol Y(z) \), see \hyperref[rhy]{\rhy}(c);
\item[(d)] \( \boldsymbol X(z) \) deviates from an analytic matrix function according to
\[
\bar\partial \boldsymbol X(z) = \boldsymbol X(z) \left(\begin{matrix} 0 & 0 \\ -\frac{w(z)\bar\partial L_{n,r}(z)}{v_i(z)} & 0 \end{matrix}\right).
\]
\end{myitemize}

One can readily verified that the following lemma holds, see \cite[Lemma~6.4]{BYa10}.

\begin{lemma}[]{}{2}
\hyperref[rhx]{\rhx} and \hyperref[rhy]{\rhy} are simultaneously solvable and the solutions are connected by \eqref{X}.
\end{lemma}

\subsection{Model Riemann-Hilbert Problem} In this subsection we present the solution  of the following Riemann-Hilbert problem (\rhn):
\begin{myitemize}
\label{rhn}
\item[(a)] \( \boldsymbol N(z) \) is analytic in \( \C\setminus[-1,1] \) and \( \lim_{z\to\infty} \boldsymbol N(z)z^{-n\sigma_3} = \boldsymbol I \);
\item[(b)] \( \boldsymbol N(z)\) has continuous traces on \( (-1,1) \) that satisfy
\[
\boldsymbol N_+(x) =\boldsymbol N_- (s) \left(\begin{matrix} 0 & \frac{\rho(x)v_i(x)}{w_+(x)} \smallskip \\ -\frac{w_+(x)}{\rho(x)v_i(x)} & 0 \end{matrix}\right);
\]
\item[(c)] \( \boldsymbol N(z) \) has the same behavior near \( \pm1 \) as \( \boldsymbol Y(z) \), see \hyperref[rhy]{\rhy}(c).
\end{myitemize}

Recall the definition of the functions \( S_i(z) \) in \eqref{Si}. Define \( S_*(z) := S_i(z) \) when \( i\in\{1,3\} \) and \( S_*(z) := \ic S_i(z) \) when \( i\in\{2,4\} \). Then it follows from \eqref{Sij} that
\[
S_{*+}(x)S_{*-}(x) = 1/v_i(x), \quad x\in(-1,1).
\]
Let \( S(z) \) and \( \varphi(z) \) be given by \eqref{S} and \eqref{phi}, respectively. It follows from \eqref{Sj} and \eqref{phij} that
\[
(S_*S\varphi^n)_-^{\sigma_3}(x)\left(\begin{matrix} 0 & \frac{\rho(x)v_i(x)}{w_+(x)} \smallskip \\ -\frac{w_+(x)}{\rho(x)v_i(x)} & 0 \end{matrix}\right)(S_*S\varphi^n)_+^{-\sigma_3}(x) = \left(\begin{matrix} 0 & 1/w_+(x)\smallskip \\ -w_+(x) & 0 \end{matrix}\right)
\]
for \( x\in(-1,1) \). It also can be readily verified with the help of \eqref{phij} that
\[
\left(\begin{matrix} 1 & \frac1{w_+(x)} \medskip \\  \frac1{2\varphi_+(x)} & \frac{\varphi_+(x)}{2w_+(x)} \end{matrix} \right) = \left(\begin{matrix} 1 & \frac1{w_-(x)} \medskip \\  \frac1{2\varphi_-(x)} & \frac{\varphi_-(x)}{2w_-(x)} \end{matrix} \right) \left(\begin{matrix} 0 & 1/w_+(x)\smallskip \\ -w_+(x) & 0 \end{matrix}\right)
\]
for \( x\in(-1,1) \). Therefore, \hyperref[rhn]{\rhn} is solved by \( \boldsymbol N(z) = \boldsymbol{CM}(z) \), where
\begin{equation}
\label{M}
\boldsymbol C := (2^nS_*S)^{-\sigma_3}(\infty) \qandq \boldsymbol M(z) := \left(\begin{matrix} 1 & \frac1{w(z)} \medskip \\  \frac1{2\varphi(z)} & \frac{\varphi(z)}{2w(z)} \end{matrix} \right) (S_*S\varphi^n)^{\sigma_3}(z).
\end{equation}

\subsection{Analytic Approximation}

To solve \hyperref[rhx]{\rhx}, we first solve its analytic version. That is, consider the following Riemann-Hilbert problem (\rha):
\begin{myitemize}
\label{rha}
\item[(a)] \( \boldsymbol A(z) \) is analytic in \( \C\setminus([-1,1]\cup\partial E_R) \) and \( \lim_{z\to\infty} \boldsymbol A(z)z^{-n\sigma_3} = \boldsymbol I \);
\item[(b,c)] \( \boldsymbol A (z)\) satisfies \hyperref[rhx]{\rhx}(b,c).
\end{myitemize}

\begin{lemma}[]{}{3}
For all \( n \) large enough there exists a matrix \( \boldsymbol Z(z) \), analytic in \( \overline\C\setminus\partial E_R \) and satisfying
\[
\boldsymbol Z(z) = \boldsymbol I + \boldsymbol O\left( R_*^{-n}\right)
\]
uniformly in \( \overline\C \) for any \( r<R_*<R \), such that \( \boldsymbol A(z) = \boldsymbol C \boldsymbol Z(z)\boldsymbol M(z) \) solves \hyperref[rha]{\rha}. 
\end{lemma}
\begin{proof}
Assume that there exists a matrix \( \boldsymbol Z(z) \) that is analytic in \( \overline\C\setminus\partial E_R \), is equal to \( \boldsymbol I \) at infinity, and satisfies
\[
\boldsymbol Z_+(s) = \boldsymbol Z_-(s) \boldsymbol M(s) \left(\begin{matrix} 1 & 0 \medskip \\ \frac{w(s)l_n(s)}{v_i(s)} & 1 \end{matrix}\right)\boldsymbol M^{-1}(s), \quad s\in \partial E_R.
\]
It can be readily verified that \( \boldsymbol A(z) = \boldsymbol C \boldsymbol Z(z)\boldsymbol M(z) \) solves \hyperref[rha]{\rha}. To show that such \( \boldsymbol Z(z) \) does indeed exist, observe that
\[
\det \boldsymbol M(z) = \frac{\varphi(z)}{2w(z)} - \frac1{2\varphi(z)w(z)} \equiv 1
\]
in the entire complex plane and that
\[
v_i(z)S_*^2(z) = (-1)^{i-1}\varphi^{k_i}(z), \quad z\not\in[-1,1],
\]
straight by the definition of \( S_i(z) \) in \eqref{Si}, where \( k_1=0 \), \( k_2=2 \), and \( k_3=k_4=1 \). Thus,
\begin{equation}
\label{smallj}
\boldsymbol M(s) \left(\begin{matrix} 1 & 0 \medskip \\ \frac{w(s)l_n(s)}{v_i(s)} & 1 \end{matrix}\right) \boldsymbol M^{-1}(s) = \boldsymbol I + \frac{(-1)^{i-1} l_n(s)}{w(s)S^2(s)\varphi^{2n+k_i}(s)} \left(\begin{matrix} \frac12\varphi(s) & -1 \medskip \\ \frac14\varphi^2(s) & -\frac12\varphi(s) \end{matrix}\right)
\end{equation}
for \( s\in \partial E_R \). It follows from the very definition of \( E_R \) that \( |\varphi(s)| = R \) for \( s\in \partial E_R \). Moreover, since \( \deg(l_n) \leq n \) and the uniform norms on \( [-1,1] \) of these polynomials are bounded by \( C_\rho^\prime \), see Proposition~\ref{prop:1}, it holds that
\[
|l_n(s)| \leq C_\rho^\prime |\varphi(s)|^n = C_\rho^\prime R^n, \quad s\in \partial E_R,
\]
by the Bernstein-Walsh inequality. Hence, we can conclude that the jump of \( \boldsymbol Z(z) \) on \( \partial E_R \) can be estimated as \( \boldsymbol I + \boldsymbol O(R^{-n}) \). It now follows from \cite[Theorem~7.18 and Corollary~7.108]{Deift} that such \( \boldsymbol Z(z) \) does exist, is unique, and has continuous traces on \( \partial E_R \) whose \( L^2 \)-norms with respect to the arclength measure are of size \( O(R^{-n}) \). This yields the desired pointwise estimate of \( \boldsymbol Z(z) \) locally uniformly in \( \overline\C\setminus\partial E_R \). Next, observe that the jump of \( \boldsymbol Z(s) \) is analytic around \( \partial E_R \) and therefore we can vary the value of \( R \). Since the solutions corresponding to different values of \( R \) are necessarily analytic continuations of each other, the desired uniform estimate follows from the locally uniform ones for any fixed \( R_*<R \) and \( R^\prime>R \).
\end{proof}

\subsection{An Auxiliary Estimate}

Denote by  \( dA \) the area measure and by \( \mathcal K \) the Cauchy area operator acting on integrable functions on \( \C \), i.e.,
\begin{equation}
\label{K}
\mathcal Kf(z) = \frac1\pi\iint \frac{f(s)}{z-s}dA.
\end{equation}
\begin{lemma}[]{}{4}
Let \( u(z) \) be a bounded function supported on \( \overline E_r \). Then
\[
\left\| \mathcal K\big(u|\varphi|^{-2n}\big) \right\| \leq C_r\frac{\log n}n \|u\|,
\]
where \( \|\cdot\| \) is the essential supremum norm and the constant \( C_r \) is independent of \( n \).
\end{lemma}
\begin{proof}
Observe that the integrand is a bounded compactly supported function and therefore its Cauchy area integral is  H\"older continuous in \( \C \) with any index \( \alpha<1 \), see~\cite[Theorem~4.3.13]{AstalaIwaniecMartin}. Moreover, since the integral is analytic in \( \overline\C\setminus\overline E_r \), the maximum of its modulus is achieved on \( \overline E_r \). Notice also that it is enough to prove the claim of the lemma only for \( u(z)=\chi_{E_r}(z) \), the indicator function of \( E_r \). 

Let \( z\in\overline E_r \). Observe that \( \varphi(s)=\tau \) when \( s=\frac12(\tau+1/\tau) \). Write \( z=\frac12(\xi+1/\xi) \). Then
\begin{eqnarray*}
\left| \mathcal K\left(\frac{\chi_{E_r}}{|\varphi|^{2n}}\right)(z) \right| & \leq & \frac1\pi\iint_{E_r}\frac1{|z-s|}\frac{dA}{|\varphi(s)|^{2n}} \\
& = & \frac1\pi\iint_{1<|\tau|<r}\frac{|\tau^2-1|^2}{|(\xi-\tau)(1-1/(\tau\xi))|}\frac{dA}{|\tau|^{2n+4}}.
\end{eqnarray*}
Partial fraction decomposition now yields 
\begin{eqnarray*}
\left| \mathcal K\left(\frac{\chi_{E_r}}{|\varphi|^{2n}}\right)(z) \right| & \leq & \frac1\pi\iint_{1<|\tau|<r} \left| \frac{\xi}{\tau-\xi} + \frac{\tau}{\tau-1/\xi} \right| \frac{|\tau^2-1|}{|\tau|^{2n+4}}dA \\
& \leq & \frac{2r^3}\pi \iint_{1<|\tau|<r} \left(\frac1{|\tau-\xi|} + \frac1{|\tau-1/\xi|} \right) \frac{dA}{|\tau|^{2n+4}}.
\end{eqnarray*}
Write \( \tau=\varrho e^{\ic\theta} \) and  \( \xi = \varrho_* e^{\ic\theta_*} \). Then
\begin{eqnarray*}
|\tau-\xi| & = & \sqrt{(\varrho-\varrho_*)^2+4\varrho\varrho_*\sin^2\left(\frac{\theta-\theta_*}2\right)} \\
& \geq & \frac1{\sqrt2}\left( |\varrho-\varrho_*| + \sqrt{\varrho\varrho_*}\left|2\sin\left(\frac{\theta-\theta_*}2\right)\right| \right) \\
& \geq & C \big( |\varrho-\varrho_*| + |\theta-\theta_*| \big)
\end{eqnarray*}
for some constant \( C<1/\sqrt2 \), where on the last step we used inequalities \( \varrho\varrho_*\geq 1 \) and \( \min_{[-\pi/2,\pi/2]}|\sin x/x|>0 \). Since \( \varrho/\varrho_*\geq 1/r \), the constant \( C \)  can be adjusted so that
\[
|\tau-1/\xi| \geq C \big( |\varrho-1/\varrho_*| + |\theta+\theta_*| \big) \geq  C \big( |\varrho-\varrho_*| + |\theta+\theta_*| \big)
\]
is true as well. By going to polar coordinates and applying the above estimates we get that
\begin{eqnarray*}
\left| \mathcal K\left(\frac{\chi_{E_r}}{|\varphi|^{2n}}\right)(z) \right| & \leq & \frac{4r^3}{\pi C} \int_1^r \left( \int_0^\pi \frac{d\theta}{|\varrho-\varrho_*|+\theta}\right) \frac{d\varrho}{\varrho^{2n+3}} \\
& = & \frac{4r^3}{\pi C} \left(\int_{I_1} + \int_{I_2}\right) \log\left(1+\frac\pi{|\varrho-\varrho_*|}\right) \frac{d\varrho}{\varrho^{2n+3}} =: S_1 + S_2,
\end{eqnarray*}
where \( I_1 = (1,r)\cap \big\{ \varrho:|\varrho-\varrho_*|<\pi/n \big\} \) and \( I_2 = (1,r)\setminus I_1 \). Then
\begin{eqnarray*}
S_1 & \leq & \frac{8r^3}{\pi C}\int_0^{\pi/n} \log\left(1+\frac\pi\varrho\right) d\varrho = \frac{8r^3}C \int_{n+1}^\infty \frac{\log tdt}{(t-1)^2} \\
&= & \frac{8r^3}C\left(\frac{\log(n+1)}n + \int_{n+1}^\infty\frac{dt}{t(t-1)}\right) \leq \frac{8r^3}C \frac{\log(n+1)+1}n.
\end{eqnarray*}
Finally, it holds that
\[
S_2 \leq \frac{8r^3\log(n+1)}{\pi C}\int_1^\infty \frac{d\varrho}{\varrho^{2n+3}} = \frac{4r^3}{\pi C}\frac{\log(n+1)}{n+1},
\]
which finishes the proof of the lemma.
\end{proof}

\subsection{\(\bar\partial\)-Problem}

Consider the following \( \bar\partial \)-problem (\pbd):
\begin{myitemize}
\label{pbd}
\item[(a)] \( \boldsymbol D(z) \) is a continuous matrix function on \( \overline\C \) and \( \boldsymbol D(\infty)=\boldsymbol I \);
\item[(b)] \( \boldsymbol D(z) \) satisfies \( \bar\partial \boldsymbol D(z) = \boldsymbol D(z) \boldsymbol W(z) \), where 
\[
\boldsymbol W(z) := \boldsymbol Z(z) \boldsymbol M(z) \left(\begin{matrix} 0 & 0 \\ -w(z)\bar\partial L_{n,r}(z)/v_i(z) & 0 \end{matrix}\right)\boldsymbol M^{-1}(z)\boldsymbol Z^{-1}(z).
\]
\end{myitemize}

Notice that \( \boldsymbol W(z) \) is supported by \( \overline E_r \) and therefore \( \boldsymbol D(z) \) is necessarily analytic in the complement of \( \overline E_r \).

\begin{lemma}[]{}{5}
The solution of \hyperref[pbd]{\pbd} exists for all \( n \) large enough and it holds uniformly in \( \overline\C \) that
\[
\boldsymbol D(z) = \boldsymbol I + \boldsymbol O( \varepsilon_n ).
\]
\end{lemma}
\begin{proof}
As explained in \cite[Lemma~8.1]{BYa10}, solving \hyperref[pbd]{\pbd} is equivalent to solving an integral equation
\[
\boldsymbol I = ( \mathcal I - \mathcal K_{\boldsymbol W} )\boldsymbol D(z)
\]
in the space of bounded matrix functions on \( \C \), where \( \mathcal I \) is the identity operator and \( \mathcal K_{\boldsymbol W} \) is the Cauchy area operator \eqref{K} acting component-wise on the product \( \boldsymbol m(s)\boldsymbol W(s) \) for a bounded matrix function \( \boldsymbol m(z) \). If \( \vb{\mathcal K_{\boldsymbol W}} \), the operator norm of \( \mathcal K_{\boldsymbol W} \), is less than \( 1-\epsilon \), \( \epsilon\in(0,1) \), then \( (\mathcal I - \mathcal K_{\boldsymbol W})^{-1} \) exists as a Neumann series and 
\[
\boldsymbol D(z) = ( \mathcal I - \mathcal K_{\boldsymbol W} )^{-1}\boldsymbol I = \boldsymbol I + \boldsymbol O_\epsilon(\vb{\mathcal K_{\boldsymbol W}})
\]
uniformly in \( \overline\C \)  (it also holds that \( \boldsymbol D(z) \) is H\"older continuous in \( \C \)). It follows from Lemma~\ref{lem:4} that to estimate \( \vb{\mathcal K_{\boldsymbol W}} \), we need to estimate \( L^\infty \)-norms of the entries of \( \boldsymbol W(z) \). To this end, similarly to \eqref{smallj}, we get that
\[
 \boldsymbol W(z) = \frac{(-1)^i\bar\partial L_{n,r}(z)}{w(z)S^2(z)\varphi^{2n+k_i}(z)} \boldsymbol Z(z) \left(\begin{matrix} \frac12\varphi(z) & -1 \medskip \\ \frac14\varphi^2(z) & -\frac12\varphi(z) \end{matrix}\right) \boldsymbol Z^{-1}(z), \quad z\in \overline E_r.
\]
Using Proposition~\ref{prop:1} and Lemma~\ref{lem:3} we can conclude that entries of \( \boldsymbol W(z) \) are continuous functions on \( \C \) supported by \( \overline E_r \) with absolute values bounded above by \( C_\rho |\varphi(z)|^{-2n}n \varepsilon_n/\log n \) for some constant \( C_\rho \) independent of \( n \). Hence, \( \vb{\mathcal K_{\boldsymbol W}} = O(\varepsilon_n) \) as claimed.
\end{proof}

\subsection{Asymptotic Formulae}

It readily follows from \hyperref[rha]{\rha} and \hyperref[pbd]{\pbd} as well as Lemmas~\ref{lem:3} and~\ref{lem:5} that \hyperref[rhx]{\rhx} is solved by
\[
\boldsymbol X(z) = \boldsymbol C\boldsymbol D(z)\boldsymbol Z(z)\boldsymbol M(z).
\]
Given a closed set \( B\subset \overline\C\setminus [-1,1] \), we can choose \( r \) amd \( R \) so that \( \overline E_R \cap B = \varnothing \). Then it holds that \( \boldsymbol Y(z) = \boldsymbol X(z) \) for \( z\in B \) by \eqref{X}. Write
\[
\boldsymbol D(z)\boldsymbol Z(z) = \boldsymbol I + \left(\begin{matrix} \upsilon_{n1}(z) & \upsilon_{n2}(z) \medskip \\ \upsilon_{n3}(z) & \upsilon_{n4}(z) \end{matrix}\right).
\]
It follows from Lemmas~\ref{lem:3} and~\ref{lem:5} that \( |\upsilon_{nj}(z)| = O( \varepsilon_n ) \) uniformly in \( \overline\C \) and that \( \upsilon_{nj}(\infty)=0 \). Then we get from \eqref{Y} and \eqref{M} that
\[
P_n(z) = \left( 1+\upsilon_{n1}(z) + \frac{\upsilon_{n2}(z)}{2\varphi(z)} \right) \frac{(S_*S)(z)}{(S_*S)(\infty)} \left( \frac{\varphi(z)}2 \right)^n, \quad z\in B.
\]
Since \( S_*(z)/S_*(\infty)=S_i(z)/S_i(\infty) \), the first claim of the theorem follows. Next, notice that the first column of \( \boldsymbol Y(z) \) is entire and is equal to the first column of
\[
\boldsymbol X_+(x) \left(\begin{matrix} 1 & 0 \smallskip \\ w_+(x)/(\rho(x)v_i(x)) & 1 \end{matrix}\right)
\]
for \( x\in [-1,1] \) by \eqref{X} and Proposition~\ref{prop:1}. Since the functions \( \upsilon_{ni}(z) \) are continuous across \( [-1,1] \) and \( S_{*\pm}(x)/S_*(\infty)=S_{i\pm}(x)/S_i(\infty) \), we deuce from \eqref{Sj}, \eqref{phij}, \eqref{Sij}, and \eqref{M} that
\begin{multline*}
P_n(x) = (1+\upsilon_{n1}(x))\frac{(S_iS\varphi^n)_+(x) + (S_iS\varphi^n)_-(x)}{2^n(S_iS)(\infty)} + \\ \upsilon_{n2}(x)\frac{(S_iS\varphi^{n-1})_+(x) + (S_iS\varphi^{n-1})_-(x)}{2^{n+1}(S_iS)(\infty)}
\end{multline*}
for any \( x\in[-1,1] \). It now follows from \eqref{theta}, \eqref{phij}, and \eqref{Sij} that
\[
(S_iS\varphi^k)_+(x) + (S_iS\varphi^k)_-(x) = \frac{2\cos\big(k\arccos(x) + \theta(x) + \theta_i(x) \big)}{\sqrt{\rho(x)|v_i(x)|}}, \quad  x\in[-1,1].
\]
The last two formulae now yield the second claim of the theorem. Finally, it is known, see \cite[Equations (9.6) and (9.7)]{KMcLVAV04}, that
\[
\begin{cases}
a_{n,i}^2 & = \displaystyle \lim_{z\to\infty} z^2[\boldsymbol Y(z)]_{12}[\boldsymbol Y(z)]_{21}, \medskip \\
b_{n,i} & = \displaystyle \lim_{z\to\infty} \big( z- P_{n+1,i}(z)[\boldsymbol Y(z)]_{22} \big),
\end{cases}
\]
where \( \boldsymbol Y(z) \) corresponds to the index \( n \). As in the first part of the proof, we get that
\[
[\boldsymbol Y(z)]_{12} = [\boldsymbol X(z)]_{12} = \frac1{w(z)}\frac{1+\upsilon_{n1}(z) + \upsilon_{n2}(z)\varphi(z)/2}{2^n(S_*S)(\infty)(S_*S)(z)\varphi^n(z)}
\]
and
\[
[\boldsymbol Y(z)]_{21} = [\boldsymbol X(z)]_{21} = \left( \upsilon_{n3}(z) + \frac{1+\upsilon_{n4}(z)}{2\varphi(z)} \right) 2^n(S_*S)(\infty)(S_*S)(z)\varphi^n(z)
\]
for all \( z \) large. Since \( \upsilon_{nj}(\infty)=0 \), it holds that
\[
a_{n,i}^2 = \frac14+ \lim_{z\to\infty} z\upsilon_{n3}(z)(1+z\upsilon_{n2}(z)) = \frac14 + O( \varepsilon_n )
\]
by the maximum modulus principle for holomorphic functions. Similarly, we have that
\[
[\boldsymbol Y(z)]_{22} = [\boldsymbol X(z)]_{22} = \left( \upsilon_{n3}(z) + \frac12(1+\upsilon_{n4}(z))\varphi(z) \right)\frac1{w(z)}\frac{2^n(S_*S)(\infty)}{(S_*S)(z)\varphi^n(z)}
\]
for all \( z \) large. Hence,
\[
P_{n+1,i}(z)[\boldsymbol Y(z)]_{22} = \frac{\varphi^2(z)}{4w(z)} \left( 1 + \upsilon_{n+11}(z) + \frac{\upsilon_{n+12}(z)}{2\varphi(z)} \right) \left( 1 + \upsilon_{n4}(z) + 2\frac{\upsilon_{n3}(z)}{\varphi(z)}\right)
\]
in this case. It can be readily verified that
\[
\frac{\varphi^2(z)}{4w(z)} = z + \frac z{2w(z)(z+w(z))} - \frac1{4w(z)} = z + O\left(\frac 1z\right)
\]
as \( z\to\infty \). Therefore,
\[
b_{n,i} = - \lim_{z\to\infty} z\big( \upsilon_{n+11}(z) + \upsilon_{n4}(z)\big) = O( \varepsilon_n )
\]
again, by the maximum modulus principle for holomorphic functions. This finishes the proof of the theorem.


\end{document}